\documentclass[11pt]{article}
\usepackage{amsmath,amssymb}

\newtheorem{propo}{{\bf Proposition}}[section]
\newtheorem{coro}[propo]{{\bf Corollary}}
\newtheorem{lemma}[propo]{{\bf Lemma}} \newtheorem{theor}[propo]{{\bf
Theorem}} \newtheorem{ex}{{\sc Example}}[section]
\newtheorem{definition}{\bf Definition}
\newenvironment{proof}{{\bf Proof.}}{$\Box$}

\def\N{{\mathbb N}}
\begin{document}
\vspace*{1.0in}

\begin{center} WEAK C-IDEALS OF LEIBNIZ ALGEBRAS
\end{center}
\bigskip

\begin{center} David A. TOWERS and  Zekiye CILOGLU
\end{center}
\bigskip

\begin{center} Department of Mathematics and Statistics

Lancaster University

Lancaster LA1 4YF

England

d.towers@lancaster.ac.uk 

and

Department of Mathematics

Suleyman Demirel University

$32260$, Isparta

TURKEY

zekiyeciloglu@sdu.edu.tr
\end{center}
\bigskip

\begin{abstract} 
A subalgebra $B$ of a Leibniz algebra $L$ is called a weak c-ideal of $L$ if
there is a subideal $C$ of $L$ such that $L=B+C$ and $B\cap C\subseteq B_{L}$
where $B_{L}$ is the largest ideal of $L$ contained in $B.$ This is
analogous to the concept of a weakly c-normal subgroup, which has been
studied by a number of authors. We obtain some properties of weak c-ideals
and use them to give some characterisations of solvable and supersolvable
Leibniz algebras generalising previous results for Lie algebras. We note that one-dimensional weak c-ideals are c-ideals, and show that a result of  Turner in \cite[Theorem 3.2.9]{bnt} classifying Leibniz algebras in which every one-dimensional subalgebra is a c-ideal is false for general Leibniz algebras, but holds for symmetric ones.
 \\[+2mm]

\end{abstract}

\section{Introduction}
An algebra $L$ over a field $F$ is called a {\em Leibniz algebra} if, for every $x,y,z \in L$, we have
\[  [x,[y,z]]=[[x,y],z]-[[x,z],y]
\]
In other words the right multiplication operator $R_x : L \rightarrow L : y\mapsto [y,x]$ is a derivation of $L$. As a result such algebras are sometimes called {\it right} Leibniz algebras, and there is a corresponding notion of {\it left} Leibniz algebras, which satisfy
\[  [x,[y,z]]=[[x,y],z]+[y,[x,z]].
\]
Clearly the opposite of a right (left) Leibniz algebra is a left (right) Leibniz algebra, so, in most situations, it does not matter which definition we use. A {\em symmetric} Leibniz algebra $L$ is one which is both a right and left Leibniz algebra and in which $[[x,y],[x,y]]=0$ for all $x,y\in L$. This last identity is only needed in characteristic two, as it follows from the right and left Leibniz identities otherwise (see \cite[Lemma 1]{jp}). Symmetric Leibniz algebras $L$ are flexible, power associative and have $x^3=0$ for all $x\in L$ (see \cite[Proposition 2.37]{feld}), and so, in a sense, are not far removed from Lie algebras.
\par
 
Every Lie algebra is a Leibniz algebra and every Leibniz algebra satisfying $[x,x]=0$ for every element is a Lie algebra. They were introduced in 1965 by Bloh (\cite{bloh}) who called them $D$-algebras, though they attracted more widespread interest, and acquired their current name, through work by Loday and Pirashvili ({\cite{loday1}, \cite{loday2}). They have natural connections to a variety of areas, including algebraic $K$-theory, classical algebraic topology, differential geometry, homological algebra, loop spaces, noncommutative geometry and physics. A number of structural results have been obtained as analogues of corresponding results in Lie algebras.
\par

The {\it Leibniz kernel} is the set $I=$ span$\{x^2:x\in L\}$. Then $I$ is the smallest ideal of $L$ such that $L/I$ is a Lie algebra.
Also $[L,I]=0$.
\par

We define the following series:
\[ L^1=L,L^{k+1}=[L^k,L]  (k\geq 1) \hbox{ and } L^{(0)}=L,L^{(k+1)}=[L^{(k)},L^{(k)}] (k\geq 0).
\]
Then $L$ is {\em nilpotent of class n} (resp. {\em solvable of derived length n}) if $L^{n+1}=0$ but $L^n\neq 0$ (resp. $ L^{(n)}=0$ but $L^{(n-1)}\neq 0$) for some $n \in \N$. It is straightforward to check that $L$ is nilpotent of class n precisely when every product of $n+1$ elements of $L$ is zero, but some product of $n$ elements is non-zero.The {\em nilradical}, $N(L)$, (resp. {\em radical}, $R(L)$) is the largest nilpotent (resp. solvable) ideal of $L$.
\par

In \cite{tc} we introduced the concept of a weak c-ideal, which is analogous to the concept of a weakly c-normal subgroup as introduced by Zhu, Guo and Shum in \cite{Zhu et al} and which has since been further studied by a number of authors. It is a generalisation of the concept of a c-ideal, as introduced in \cite{david 2}. Here we investigate the extent to which the results in \cite{tc} can be extended to Leibniz algebras.
\par

In section two we give some basic properties of weak c-ideals; in particular, it is shown that weak c-ideals inside the Frattini subalgebra of
a Leibniz algebra $L$ are necessarily ideals of $L$. In section three we first show that all maximal subalgebras of $L$ are weak c-ideals of $L$ if and
only if $L$ is solvable and that $L$ has a solvable maximal subalgebra that is a weak c-ideal if and only if $L$ is solvable. Unlike the corresponding
results for c-ideals, it is necessary to restrict the underlying field to characteristic zero, as was shown by an example in \cite{tc}. Finally we have that if all
maximal nilpotent subalgebras of $L$ are weak c-ideals, or if all Cartan subalgebras of $L$ are weak c-ideals and $F$ has characteristic zero, then $
L $ is solvable.

In section four we show that if $L$ is a solvable symmetric Lie algebra over a general field and every maximal subalgebra of each maximal nilpotent subalgebra of $
L $ is a weak c-ideal of $L$ then $L$ is supersolvable. If each of the maximal nilpotent subalgebras of $L$ has dimension at least two then the
assumption of solvability can be removed. Similarly if the field has characteristic zero and $L$ is not three-dimensional simple then this
restriction can be removed. 
\par

In the final section we see that every one-dimensional subalgebra is a weak c-ideal if and only if it is a c-ideal, and go on to study the class of Leibniz algebras in which every one-dimensional subalgebra is a c-ideal. It is shown that the cyclic subalgebras in this class are at most two dimensional. There is a characterisation of all of the algebras in this class given by Turner in \cite[Theorem 3.2.9]{bnt}, but an example is given to show that this result is false. A number of properties of such algebras are given, but a full classification appears complicated. However, it is shown finally that Turner's result does hold for symmetric Leibniz algebras.
\par

Throughout, the term 'Leibniz algebra' will refer to a finite-dimensional right Leibniz algebra over a field $F$. If $A,B$ are subalgebras of $L$ with $A\subseteq B$, the {\em centraliser} of $A$ in $B$, $C_B(A)=\{b\in B \mid [b,A]+[A,b]=0\}$. The {\em normaliser} of $A$ in $B$, $N_B(A)=\{b\in B \mid [b,A]+[A,b]\subseteq A\}$. Algebra direct sums will be denoted by $\oplus$, whereas direct sums of the vector space structure alone will be written as $\dot{+}$. Subsets will be denoted by `$\subseteq$' and proper subsets by `$\subset$'.

\section{Preliminary Results}

\begin{definition} \cite{msy} Let $L$ be a Leibniz algebra over a field $F$ and $B$ be a subalgebra of $L$. We call $B$ a subideal of $L$ if there is a chain of subalgebras $$B=B_{t} \subset B_{t-1} \subset...\subset B_{0}=L$$ such that $B_{i}$ is an ideal of $B_{i-1}$ for each $1 \leq i \leq t.$
\end{definition}

\begin{definition} \cite{msy}
Let $L$ be a Leibniz algebra and $H$ a subalgebra of $L$. Then $H$ is called a c-ideal of $L$ if there is an ideal $K$ of $L$ such that $L=H+K$ and  $H\cap K$ is contained in the core of $H$ (with respect to $L$), denoted by $H_{L}$, where this is the largest ideal of $L$ contained in $H$.
\end{definition}

\begin{definition}
A subalgebra $B$ of a Leibniz algebra $L$ is a weak c-ideal of $L$ if there exists a subideal $C$ in $L$ such that $ L=B+C $ and $B\cap C\subseteq B_{L}$.
\end{definition}

\begin{definition}
A Leibniz algebra $L$ is called weakly c-simple if $L$ does not contain any weak c-ideals other than $L$, the trivial subalgebra $0$ and the Leibniz kernel $I.$ It is simple if these same three subalgebras are the only ideals of $L$.
\end{definition}

\begin{lemma}\label{1}
Let $L$ be a Leibniz algebra. Then the following statements are valid:
\begin{enumerate}
\item Let $ B $ be a subalgebra of $L$. If $B$ is a c-ideal of $L$ then $B$ is a weak c-ideal of $L.$
\item $ L $ is weakly c-simple if and only if $ L $ is simple.
\item If $B$ is a weak c-ideal of $L$ and $K$ is a subalgebra with $B\subseteq  K\subseteq L$ then $B$ is a weak c-ideal of $K$.
\item If $A$ is an ideal of $L$ and $A\subseteq B$ then $B$ is a weak c-ideal of $L$ if and only if $B/A$ is a weak c-ideal of $L/A.$
\end{enumerate}
\end{lemma}
\begin{proof}
\begin{enumerate}
\item This is apparent from the definition.
\item Suppose first that $L$ is simple and let $B$ be a weak c-ideal with $B\neq L$. Then by definition of weak c-ideal $$L=B+C \hbox{ and } B\cap C\subseteq B_{L}$$ where $C$ is a subideal of $L.$ But since $L$ is simple $B_{L}$ must be $0$ or $I.$ If $B_{L}=0$ then $C=L$, since $C\neq 0$ and is a subideal of $L$, whence $B=0$. If $B_{L}=I$ then, similarly, $C=L$ and hence $B=I$. Hence $L$ is weakly c-simple.
\par
Conversely, suppose $L$ is weakly c-simple. Then, since every ideal of $L$ is a weak c-ideal, $L$ must be simple.
\item If $B$ is a weak c-ideal of $L$, then there exist a subideal $C$ of $L$ such that 
\begin{equation*}
L=B+C\text{ and }B\cap C\subseteq B_{L}
\end{equation*}%
Now $K=K\cap L=K\cap \left( B+C\right) =B+\left( K\cap C\right).$
Since $C$ is a subideal of $L$ there exists a chain of subalgebras
\begin{equation*}
C=C_{n}\subset C_{n-1}\subset ...\subset C_{0}=L
\end{equation*}%
where $C_{j}$ is an ideal of $C_{j-1}$ for each $1\leq j\leq n.$ If we
intersect each term in this chain with $K$ we get
\begin{equation*}
C\cap K=C_{n}\cap K\subseteq C_{n-1}\cap K\subseteq \ ...\subseteq C_{0}\cap K=L\cap K=K
\end{equation*}%
and obviously $C_{j}\cap K$ is an ideal of $C_{j-1}\cap K$ for each $0\leq
j\leq n.$ Thus $C\cap K$ is a subideal of $K.$ Also,%
\begin{equation*}
B\cap \left( C\cap K\right) \subseteq B_{K}
\end{equation*}%
and so that $B$ is a $weak\ c$-$ideal$ of $K.$

(4) Suppose first that $B/A$ is a $weak\ c$-$ideal$ of $L/A.$
Then there exists a subideal $C/A$ of $L/A$ such that%
\begin{equation*}
\frac{L}{A}=\frac{B}{A}+\frac{C}{A} \text{ and } \frac{B}{A}\cap \frac{C}{A}\subseteq
\left(\frac{B}{A}\right) _{L/A}=\frac{B_{L}}{A}
\end{equation*}%
It follows that $L=B+C$ and $B\cap C\subseteq B_{L}$ where $C$ is a subideal of $%
L.$ 
\par

Suppose conversely that $A$ is an ideal of $L$ with $A\subseteq B$ and $B$ is
a $weak\ c$-$ideal$ of $L.$ Then there exists a $C$ subideal of $L$ such that%
\begin{equation*}
L=B+C\text{ and }B\cap C\subseteq B_{L}.
\end{equation*}%
Since $A$ is an ideal and $A\leq B$ the factor algebra%
\begin{equation*}
\frac{L}{A}=\frac{ B+C}{A}=\frac{B}{A}+\frac{C+A}{A}
\end{equation*}%
where $(C+A) /A$ is a subideal of $L/A$ and
\begin{equation*}
\frac{B}{A} \cap \frac{C+A}{A}=\frac{B\cap (C+A)}{A}=\frac{A+B\cap C}{A}\subseteq \frac{B_{L}}{ A}=\left(\frac{B}{A}\right) _{L/ A}
\end{equation*}%
so $B/A$ is a $weak\ c$-$ideal$ of $L/ A.$
\end{enumerate}
\end{proof}
\medskip

The \textit{Frattini subalgebra} $F(L)$ of a Leibniz algebra $L$ is the intersection of all of the maximal subalgebras of $L.$ The \textit{Frattini ideal}, $F(L)_L$, of $L$ is denoted by $\phi(L).$

The next result is a generalisation of \cite[Proposition 2.2]{david 2} and the same proof works, but we will include it for completeness.

\begin{propo}
Let $B$, $C$ be subalgebras of $L$ with $B\subseteq F(C).$ If $B$ is a weak c-ideal of $L,$ then $B$ is an ideal of $L$ and $B\subseteq\varphi (L).$
\end{propo}
\begin{proof}
Suppose that $L=B+K$ where $K$ is a subideal of $L$ and $B\cap K\subseteq B_{L}.$ Then 
$$C=C\cap L=C\cap (B+K)=B+C\cap K=C\cap K$$ since $B\subseteq F(C).$ Hence, $B\subseteq C\subseteq K,$ giving $B=B\cap K\subseteq B_{L}$ and $B$ is an ideal of $L.$ It then follows from \cite[Lemma 4.1]{david 1} that $B\subseteq \varphi (L).$
\end{proof} 
\medskip

 An ideal $A$ is complemented in $L$ if there is a subalgebra $U$ of $L$ such that $L=A+U$ and $A\cap U=0.$ We adapt this to define a subideal complement as follows:
\begin{definition}
Let $L$ be a Leibniz algebra and let $B$ be a subalgebra of $L.$ Then $B$ has a subideal complement in $L$ if there is a subideal $C$ of $L$ such that $L=B+C$ and $B\cap C=0.$
\end{definition}

Then we have the following lemma:

\begin{lemma}
If $B$ is a weak c-ideal of a Leibniz algebra $L,$ then $ B/B_{L}$ has a subideal complement in $L/B_{L}.$ Conversely, if $B$ is a subalgebra of $L$ such that $ B/B_{L}$ has a subideal complement in $L/B_{L},$ then $B$ is a weak c-ideal of $L.$
\end{lemma}
\begin{proof}
Let $B$ be a weak c-ideal of $L$. Then there exists a subideal $C$ of $L
$ such that $B+C=L$ and $B\cap C\subseteq B_{L}.$ If $B_{L}=0$ then $B\cap C=0$
and $C$ is a subideal complement of $B$ in $L.$  So, assume that $
B_{L}\neq 0,$ then we can construct the factor algebras $B/ B_{L}$ and 
$( C+B_{L}) /B_{L}.$ If we intersect these two factor algebras;
\begin{eqnarray*}
\frac{B}{ B_{L}} \cap \frac{C+B_{L}}{B_{L}} 
&=&\frac{B\cap ( C+B_{L})}{ B_{L}} \\
&=&\frac{B_{L}+(B\cap C)}{ B_{L}} \\
&\subseteq &\frac{B_{L}}{ B_{L}}=0
\end{eqnarray*}%
Hence, $(C+B_{L})/ B_{L}$ is a subideal complement of $B/ B_{L}$
in $L/ B_{L}.$ Conversely, if $K$ is a subideal of $L$ such that $%
K/B_{L}$ is a subideal complement of $B/ B_{L}$ in $L/B_{L}$ then we have that
\begin{equation*}
\frac{L}{B_{L}}=\frac{B}{B_{L}} + \frac{K}{B_{L}} 
\text{ and } \frac{B}{ B_{L}} \cap \frac{ K}{B_{L}} =0
\end{equation*}%
Then $L=B+K$ and $B\cap K=0\subseteq B_{L}$, whence $B$ is a weak c-ideal
of $L.$
\end{proof}

\section{Some characterisations of solvable algebras}
\begin{theor}
Let $L$ be a Leibniz algebra over a field $F$ of characteristic zero and let $B$ be an ideal of $L$. Then $B$ is solvable if and
only if every  maximal subalgebra of $L$ not containing $B$ is a weak $c$-ideal of $L.$
\end{theor}
\begin{proof} Suppose first that $B$ is solvable and let $M$ be a maximal subalgebra of $L$ not containing $B$. Then there exists $k\in {\mathbb N}$ such that $B^{(k+1)}\subseteq M$, but $B^{(k)}\not \subseteq M$. Clearly $L=M+B^{(k)}$ and $B^{(k)}\cap M$ is an ideal of $L$, so $B^{(k)}\cap M\subseteq M_L$. It follows that $M$ is a $c$-ideal and hence a weak $c$-ideal of $L$.
\par

Conversely, suppose every maximal subalgebra of $L$ not containing $B$ is a weak $c$-ideal of $L$. Let $M/I$ be a maximal subalgebra of $L/I$ not containing $(B+I)/I$. Then $M$ is a maximal subalgebra of $L$ not containing $B$, and so is a weak $c$-ideal of $L$. Hence $M/I$ is a weak $c$-ideal of $L/I$, by Lemma \ref{1}. Since $L/I$ is a Lie algebra, it follows from \cite[Theorem 3.2]{tc} that $(B+I)/I$ is solvable. It follows that $B+I$, and hence $B$, is solvable.
\end{proof}

\begin{coro}\label{c:max} Let $L$ be a Leibniz algebra over a field $F$ of characteristic zero. Then $L$ is solvable if and only if every maximal subalgebra of $L$ is a weak $c$-ideal of $L$.
\end{coro}

\begin{lemma}\label{l:solvsub} Let $L=U+C$ be a Leibniz algebra, where $U$ is a solvable subalgebra of $L$ and $C$ is a subideal of $L$. Then there exists $n_0\in {\mathbb N}$ such that $L^{(n_0)}\subseteq C$.
\end{lemma}
\begin{proof} This is the same as for \cite[Lemma 3.5]{tc}.
\end{proof}

\begin{theor}\label{t:solv}
Let $L$ be a Leibniz algebra over a field $F$ of characteristic zero. Then $L$ has a solvable maximal subalgebra that is a weak c-ideal of $L$ if and only if $L$ is solvable.
\end{theor}
\begin{proof} Suppose first that $L$ has a solvable maximal subalgebra $M$ that is a weak c-ideal of $L$. We show that $L$ is solvable. Let $L$ be a minimal counter-example. Then there is a subideal $K$ of $L$ such that $L = M + K$ and $M \cap K \subseteq M_L$. If $M_L\neq 0$ then $L/M_L$ is solvable, by the minimality assumption, and $M_L$ is solvable, whence $L$ is solvable, a contradiction. It follows that $M_L=0$ and $L = M \dot{+} K$. If $R$ is the solvable radical of $L$ then $R \subseteq M_L = 0$, so $L$ is a semisimple Lie algebra, by \cite[Theorem 1]{levi}. But now, for all $n\geq 1$, $L=L^{(n)}\subseteq K \neq L$, by Lemma \ref{l:solvsub}, a contradiction. The result follows.
\par
The converse follows from Corollary \ref{c:max}.
\end{proof}

\begin{theor} Let $L$ be a Leibniz algebra over a field of characteristic zero such that all maximal nilpotent subalgebras are weak $c$-ideals of $L$. Then $L$ is solvable.
\end{theor}
\begin{proof} Suppose that $L$ is not solvable but that all maximal nilpotent subalgebras of $L$ are weak $c$-ideals of $L$. Let $L=R\dot{+} S$ be the Levi decomposition of $L$, where $S\neq 0$ is a semisimple Lie algebra (\cite[Theorem 1]{levi}). Let $B$ be a maximal nilpotent subalgebra of $S$ and $U$ be a maximal nilpotent subalgebra of $L$ containing it. Then there is a subideal $C$ of $L$ such that $L=U+C$ and $U\cap C\subseteq U_L$. It follows from Lemma \ref{l:solvsub} that $S=S^{(n_0)}\subseteq L^{(n_0)}\subseteq C$, and so $B\subseteq U\cap C\subseteq U_L$, whence $S\cap U_L\neq 0$. But $S\cap U_L$ is an ideal of $S$ and so is semisimple. Since $U$ is nilpotent this is a contradiction.
\end{proof}

\begin{definition} A Cartan subalgebra of a Leibniz algebra $L$ is a nilpotent subalgebra $C$ such that $C=N_L(C)$. Over a field of characteristic zero such subalgebras certainly exist (see \cite[Section 6]{barnes}).
\end{definition}

The following result is a generalisation of a result of Dixmier in \cite{dixmier}.

\begin{lemma}\label{cartan} Let $L$ be a Leibniz algebra over a field of characteristic zero with non-zero Levi factor $S$. If $H$ is a Cartan subalgebra of $S$ and $B$ is a Cartan subalgebra of its centralizer in the the solvable radical of $L$, then $H+B$ is a Cartan subalgebra of $L$.
\end{lemma}
\begin{proof} Let $R$ be the solvable radical of $L$. Then $(H+B)^r=H^r+B^r$ for all $r\geq 1$, so $H+B$ is a nilpotent subalgebra of $L$. Let $x\in N_L(H+B)$ and put $x=s+r$ where $s\in S$ and $r\in R$. Then
\[ [x,H]+[H,x]=[s,H]+[H,s]+[r,H]+[H,r]\subseteq H+B,
\] so $[H,s]+[s,H]\subseteq H$, whence $s\in N_S(H)=H$. Moreover, 
\[  [x,B]+[B,x]=[s,B]+[B,s]+[r,B]+[B,r]\subseteq H+B,
\] so $[B,r]+[r,B]\subseteq B$, since $[s,B]+[B,s]=0$, whence $r\in N_R(B)=B$.Thus, $N_L(H+B)=H+B$ and $H+B$ is a Cartan subalgebra of $L$.
\end{proof}

\begin{theor} Let $L$ be a Leibniz algebra over a field of characteristic zero in which every Cartan subalgebra of $L$ is a weak $c$-ideal of $L$. Then $L$ is solvable.
\end{theor}
\begin{proof}
Suppose that every Cartan subalgebra of $L$ is a weak $c$-ideal of $L$ and that $L$ has a non-zero Levi factor $S.$ Let $H$ be a Cartan subalgebra of $S$ and let $B$ be a Cartan subalgebra of its centralizer in the solvable radical of $L.$ Then $C=H+B$ is a Cartan subalgera of $L$, by Lemma \ref{cartan}, and there is a subideal $K$ of $L$ such that $L=C+K$ and $ C \cap K \subseteq C_{L} $. Now there is an $n_{0}\geq 2$ such that $L^{(n_{0})} \subseteq K$ by Lemma 3.3. But $S \subseteq L^{(n_{0})} \subseteq K$ and so $C \cap S\subseteq C\cap K \subseteq C_{L}$ giving 
$C \cap S \subseteq C_{L} \cap S=0,$ a contradiction. It follows that $S=0$ and, hence, that $L$ is solvable.
\end{proof}

\section{Some characterisations of supersolvable algebras}
In this section we will restrict attention to symmetric Leibniz algebras. We know of no examples of a Leibniz algebra which is not symmetric and for which the results are false, but have been unable to establish them in that more general case. First we need the following lemma which holds in the general case.

\begin{lemma}\label{max}\cite[Lemma 5.1.2]{bnt} Let $L$ be a Leibniz algebra over any field, let $A$ be an ideal of $L$ and $U/A$ be a maximal nlpotent subalgebra of $L/A$. Then, $U=C+A$ where $C$ is a maximal nilpotent subalgebra of $L.$
\end{lemma}

\begin{theor}\label{sup}
Let $L$ be a solvable symmetric Leibniz algebra over any field $F$ in which every maximal subalgebra of each maximal nilpotent subalgebra of $L$ is a weak c-ideal of $L.$ Then $L$ is supersolvable.
\end{theor}

\begin{proof}
Let $L$ be a minimal counter-example. If $I=0$ the result follows from \cite[Theroem 4.5]{tc}, so suppose that $I\neq 0$ and let $A$ be a minimal ideal of $ L$ contained in $I$. Since $[L,I]=[I,L]=0$, $\dim A=1$. Let $U/A$ be a maximal nilpotent subalgebra of $L/A$ and let $B/A$ be a maximal subalgebra of $U/A$. Then $U=C+A$ where $C$ is a maximal nilpotent subalgebra of $L$, by Lemma \ref{max}. But $C+A$ is nilpotent, so $A\subseteq C=U$. Hence, $B$ is a maximal subalgebra of $C$ and there is a subideal $K$ of $L$ such that $L=B+K$ and $B\cap K\subseteq B_L$. Now
\[ \frac{L}{A}=\frac{B}{A}+\frac{K+A}{A} \hbox{ and } \frac{K+A}{A} \hbox{ is a subideal of }  \frac{L}{A}.
\] Moreover, 
\[ \frac{B}{A}\cap \frac{K+A}{A}= \frac{A+B\cap K}{A}\subseteq \frac{A+B_{L}}{A}\subseteq \left( \frac{B}{A}\right) _{L/A}. 
\] It follows that $L/A$ satisfies the same hypothesis as $L$, and so $L/A$ is supersolvable, by the minimality of $L$. Hence, $L$ is supersolvable.
\end{proof}
\medskip

If $L$ has no one-dimensional maximal nilpotent subalgebras, we can remove
the solvability assumption from the above result provided that $F$ has
characteristic zero.

\begin{coro}\label{cor}
Let $L$ be a symmetric Leibniz algebra over a field $F$ of characteristic zero in which every maximal nilpotent subalgebra has dimension at least two. If every maximal subalgebra of each maximal nilpotent subalgebra of $L$ is a weak c-ideal of $L,$ then $L$ is supersolvable.
\end{coro}

\begin{proof}
Let $N$ be the nilradical of $L,$ and let $x\notin N.$ Then $x\in C$ for some maximal nilpotent subalgebra $C$ of $L.$ Since $\dim C>1$,
there is a maximal subalgebra $B$ of $C$ with $x\in B.$ Then there is a subideal $K$ of $L$ such that $L=B+K$ and $B\cap K\subseteq B_{L}\leq
C_{L}\leq N$. Clearly, $x\notin K,$ since otherwise $x\in B\cap K\leq N.$
\par

Now $L/I=(B+I)/I+(K+I)/I$ where $(B+I)/I$ is nilpotent and $(K+I)/I$ is a subideal of $L/I$. It follows from \cite[Lemma 4.2]{tc} that $L^r\subseteq K+I$ for some  $r\in {\mathbb{N}}$. Hence $L^{r+1}\subseteq K$.
We have shown that if $x\notin N$ there is a subideal $K$ of $L$ with $%
x\notin K$ and $L^{r+1}\subseteq K$. Suppose that $L$ is not solvable. Then
there is a semisimple Levi factor $S$ of $L$. Choose $x\in S$. Then $x\in
S=S^{r+1}\subseteq K$, a contradiction. Thus $L$ is solvable and the result
follows from Theorem \ref{sup}.
\end{proof}
\medskip

If $L$ has a one-dimensional maximal nilpotent subalgebra, then we can also
remove the solvability assumption from Theorem 4.4., provided that
underlying field $F$ has again characteristic zero and $L$ is not
three-dimensional simple.

\begin{coro}
Let $L$ be a symmetric Leibniz algebra over a field $F$
of characteristic zero. If every maximal subalgebra of each maximal
nilpotent subalgebra of $L$ is a weak $c$-ideal of $L,$ then $L$ is
supersolvable or three dimensional simple.
\end{coro}

\begin{proof}
If every maximal nilpotent subalgebra of $L$ has dimension at
least two, then $L$ is supersolvable by Corollary \ref{cor}. So we need only
consider the case where $L$ has a one-dimensional maximal nilpotent
subalgebra say $Fx$. Suppose first that $I=0$l. Then $L$ is a Lie algebra and the result follows from \cite[Corollary 4.7]{tc}. 
\par

So now let $I\neq 0$ and let $L$ be a minimal-counter-example. Then $L$ has a minimal ideal $A\subseteq I.$  As in the proof of Theorem \ref{sup}, $L/A$ satisfies the same hypothesis as $L$ and so is supersolvable or
three-dimensional simple. In the former case, $L$ is solvable and so is
supersolvable, by Theorem \ref{sup}. In the latter case, $L=A\oplus S$ where $S$
is three-dimensional simple, by Levi's Theorem. But now $L$ is a Lie algebra and the result follows again from  \cite[Corollary 4.7]{tc}.  
\end{proof}

\section{Leibniz algebras in which every one-dimensional subalgebra is a weak c-ideal}
\begin{propo}\label{one} For a one-dimensional subalgebra $Fx$ of a Leibniz algebra $L$ the following are equivalent:
\begin{itemize}
\item[(i)] $Fx$ is a weak c-ideal of $L$;
\item[(ii)] $Fx$ is a c-ideal of $L$; and
\item[(iii)] either $Fx$ is an ideal of $L$, or there is an ideal $B$ of $L$ such that $L=B\dot{+}Fx$ and $x\notin L^2$.
\end{itemize}
\end{propo}
\begin{proof} (i) and (ii) are equivalent since a subideal of codimension one in $L$ is an ideal.
\par

If (ii) holds, then there is an ideal $B$ in $L$ such that $L=Fx+B$, and $Fx\cap B\subseteq (Fx)_L=0$ or $Fx$. The former imples that $L=B\dot{+}Fx$ and $x\notin L^2\subseteq B$; the latter implies that $Fx$ is an ideal of $L$. Hence (iii) holds. The converse is clear.
\end{proof}

\begin{definition} Put $J=\{x\in L \mid x^2=0\}$. Note that $I\subseteq J$.
\end{definition}

\begin{coro}\label{two}Let $L$ be a Leibniz algebra over any field. Then every one-dimensional subalgebra is a $c$-ideal if and only if  $L^2\cap J\subseteq Asoc(L)$.
\end{coro}
\begin{proof} Clearly, if $Fx$ is a one-dimensional subalgebra of $L$, then $x\in J$. It follows from Proposition \ref{one} that if every one-dimensional subalgebra of $L$ is a $c$-ideal, then $L^2\cap J\subseteq Asoc(L)$. 
\par

So suppose that $L^2\cap J\subseteq Asoc(L)$, and let $Fx$ be a one-dimensional subalgebra of $L$. If $x\in L^2$, then $Fx$ is an ideal of $L$. If $x\notin L^2$, let $B$ be a subspace containing $L^2$ which is complementary to $Fx$. Then $B$ is an ideal of $L$ and $L=Fx\dot{+} B$. hence, $Fx$ is a $c$-ideal, by Proposition \ref{one} (iii).
\end{proof}

\begin{propo}\label{three} Let $L=\langle x \rangle$ be a cyclic Leibniz algebra. Then every one-dimensional subalgebra of $L$ is a c-ideal if and only if $\dim L \leq 2$.
\end{propo}
\begin{proof} Suppose that every one-dimensional subalgebra of $L$ is a c-ideal and that $\dim L>1$. If $Fx^2$ is not an ideal of $L$, there is an ideal $B$ of $L$ such that $L=Fx^2+B$. But then $x+\lambda x^2 \in B$ for some $\lambda \in F$, whence $x^2=[x,x+\lambda x^2]\in B$, a contradiction. Thus $Fx^2$ is an ideal of $L$, $x^3=\lambda x^2$ for some $\lambda \in F$ and $L=Fx+Fx^2$. 
\par

Suppose conversely that $\dim L=2$. Then $L=Fx+Fx^2$, where $x^3=0$ or $x^3=x^2$. In the former case, the only one-dimensional subalgebra is $Fx^2$ and that is an ideal of $L$. In the latter case, the one-dimensional subalgebras are $Fx^2$, which is an ideal, and $F(x-x^2)$, which is complemented by $Fx^2$. 
\end{proof}

In \cite{bnt} the following result appears.

\begin{theor}\label{four} Let $L$ be a Leibniz algebra over any field $F$ . Then all one-dimensional subalgebras of $L$ are $c$-ideals of L if and only if:
\begin{itemize}
\item[(i)] $L^3=0$; or
\item[(ii)] $L=A\oplus B$, where $A$ is an abelian ideal of $L$ and $B$ is an almost abelian ideal of $L$.
\end{itemize}
\end{theor}
\begin{proof} See \cite[Theorem 3.2.9, page 26]{bnt}.
\end{proof}
\medskip

Turner defines a subalgebra $B$ of a Leibniz algebra $L$ to be {\em almost abelian} if  $B=Fx\dot{+} D$ where, $D$ is abelian and $[d,x]\subseteq D$ for all $d\in D$. (She actually has the products the other way around as she is dealing with left Leibniz algebras, whereas, here we are concerned with right Leibniz algebras.) However, this definition is problematic as it appears to be assumed in the proof that $[d,x]=d$ for all $d\in D$, and that does not follow from the definition. Also, nothing is said about $[x,d]$. Elsewhere in the literature there have been defined two types of almost abelian Leibniz algebras: $B=Fx\dot{+} D$ is called an {\em almost abelian Lie algebra} if $[d,x]=-[x,d]=d$ for all $d\in D$, and is an {\em almost abelian non-Lie} Leibniz algebra if $[d,x]=d$ for all $d\in D$, all other products being zero in each case.
\par

Moreover, the result is false, as the following example shows.

\begin{ex} Let $L$ be the three-dimensional Leibniz algebra over a field of characteristic different from $2$ with basis $a,b,x$ and non-zero products $a^2=b$, $[a,x]=-[x,a]=\frac{1}{2} a$, $[b,x]=b$. Then $(\alpha a+\beta b+\gamma x)^2=0$ if and only if $\alpha^2=-\beta \gamma$. If $\alpha=0$ then, either $\beta-0$, in which case $Fx$ is complemented by the ideal $Fa+Fb$, or $\gamma=0$, in which case $Fb$ is an ideal of $L$. If $\alpha\neq 0$, then $F(\alpha a +\beta b+ \frac{\alpha^2}{\beta} x)$ is complemented by the ideal $Fa+Fb$.  It follows that every one-dimensional subalgebra is a $c$-ideal. However, $L$ is not of the form given in Theorem \ref{four}.
\end{ex}
\medskip

In fact, the structure of Leibniz algebras in which all one-dimensional subalgebras are $c$-ideals can be more complicated than is claimed by Theorem \ref{four}. The best that we can achieve currently is the following. 

\begin{lemma}\label{five} Let $L$ be a Leibniz algebra in which every one-dimensional subalgebra is a c-ideal. Then 
\begin{itemize}
\item[(i)] all minimal abelian ideals are one dimensional;
\item[(ii)] if $Asoc(L)=Fa_1\oplus \ldots \oplus Fa_r$ and $x\in L$, then $a_i\in Z(L)$ or $[a_i,x]=\lambda_x a_i$ for some $0\neq\lambda_x\in F$, $1\leq i\leq r$.
\item[(iii)] Let $Asoc(L)=Z(L)\oplus D$ where $[a,x]=\lambda_x a$ for all $a\in D$, $x\in L$. Then, either $D=0$ or $L= (Z(L)\oplus D)\dot{+} C \dot{+} Fx$ where $[D,C]=[C,D]=0$, $(Z(L)\oplus D)\dot{+} C$ is an ideal of $L$ and $[a,x]=a$ for all $a\in D$.
\end{itemize}
\end{lemma}
\begin{proof} \begin{itemize}\item[(i)] Let $A$ be a minimal abelian ideal of $L$ and let $a\in A$. If $A\neq Fa$ then there is an ideal $K$ of $L$ such that $L=Fa\dot{+} K$. But $A\cap K=0$ so $A=Fa$, a contradiction.
\item[(ii)] Suppose that $[a_i,x]=\lambda a_i$ and $[a_j,x]=\mu a_j$ where $\lambda\neq \mu$. Then $F(a_i+a_j)$ is not an ideal of $L$ and so there is an ideal $M$ of $L$ such that $L=F(a_i+a_j)\dot{+} M$. Clearly, one of $a_i$ and $a_j$ does not belong to $M$. Suppose that $a_i\notin M$, so $L=Fa_i\oplus M$ and $a_i\in Z(L)$.
\item[(iii)] Let $\Lambda : L \rightarrow F$ be given by $\Lambda (x) = \lambda_x$. This is a linear transformation. Hence, either Im $\Lambda = 0$, in which case $D = 0$, or else $L =$ Ker $\Lambda \dot{+} Fx$ and $\lambda_x = 1$. Put $L = (Z \oplus D) \dot{+} C \dot{+} Fx$, where $C \subseteq$ Ker $\Lambda$. Let $d\in D$, $c\in C$ with $[c,d]=\lambda d$. Then $$\lambda^2 d=\lambda [c,d]=[c[c,d]]=[c^2,d]-[[c,d],c]=0,$$ so $\lambda=0$. Hence $[D,C]=[C,D]=0$. It is straightforward to check that Ker $\Lambda$ is an ideal of $L$.
\end{itemize}
\end{proof}
\medskip

However, we can retrieve Theorem \ref{four} for symmetric Leibniz algebras.

\begin{theor}\label{symm} Let $L$ be a symmetric Leibniz algebra over any field $F$ . Then all one-dimensional subalgebras of $L$ are $c$-ideals of $L$ if and only if:
\begin{itemize}
\item[(i)] $L^3=0$; or
\item[(ii)] $L=A\oplus B$, where $A$ is an abelian ideal of $L$ and $B$ is an almost abelian Lie ideal of $L$.
\end{itemize}
\end{theor}
\begin{proof} Suppose that  all one-dimensional subalgebras of $L$ are $c$-ideals of L. First note that, if $x,y\in L$, then $[x,y]^2\in L^2\cap J$, so $L^2\subseteq Asoc(L)$, by Corollary \ref{two}. Also, $L$ must have the structure given in Lemma \ref{five} (iii). If $D=0$ then $L^2\subseteq Z(L)$ and (i) holds. So suppose that $D\neq 0$.
\par

Let $0\neq a\in D$. If $z\in Z$, then $F(z+a)$ is not an ideal of $L$ and so $z+a\not \in L^2$. But $a=[a,x]\in L^2$, so $z\not \in L^2$ and $Z\cap L^2=0$. It follows that $L^2=D$. Now, if $c\in C$, we have that $c^2\in L^2=D$, so $c^3=c^2$. But $c^3=0$, by \cite[Proposition 2.17]{feld}. If $[c,x]=0$, then $[x,c]\in L^2=D$ and so $[x,c]=[[x,c],x]=[x,[c,x]]=0$, since $L$ is also a flexible algebra, by  \cite[Proposition 2.17]{feld} again. Hence $Fc$ is an ideal of $L$. But then $c\in L^2\cap C=D\cap C=0$.
\par

So suppose that $[c,x]\neq 0$. Then $[c-[c,x],x]=0$ and $[x,c-[c,x]]=[x,c]-[x,[c,x]]=[x,c]-[[x,c],x]=0$, using the flexible law again. Also, $[c-[c,x],c-[c,x]]=-[c,[c,x]]-[[c,x],c]=-[c,[c,x]]-[c,[x,c]]=0$, since $[c,x]+[x,c]\in I$. It follows that $F(c-[c,x])$ is an ideal of $L$ and hence is inside $L^2=D$. Thus, $c\in C\cap D=0$, so $C=0$.
\par

Finally, for all $a\in D$, $[a,x]=a$. Also, $[a,x]+[x,a]\in I$, so $0= [[a,x],x]+[[x,a],x]=[a,x]+[x,a]$ and $[x,a]=-a$. Thus, $L$ is as described in (ii).
\end{proof}


\begin{thebibliography}{1}
\bibitem{barnes} D.W. BARNES, `Some theorems on Leibniz algebras', {\em Comm. Alg.} {\bf 39} (2011), 2463-2472.
\bibitem{levi} D.W. BARNES, `On Levi's Theorem for Leibniz algebras', {\em Bull. Aust. Math. Soc.} {\bf 86} (2012), 184-185.
\bibitem{bloh} A. BLOH. `On a generalization of the concept of Lie algebra'. {\em Dokl. Akad. Nauk SSSR.} {\bf 165} (1965), 471--473.
\bibitem{bur} T. BURCH `Supersolvable Leibniz Algebras'. PhD thesis. North Carolina State University, 2015.
\bibitem{dixmier} J. DIXMIER, `Sous-algebres de Cartan et decompositions de Levi dans les algebres de Lie (French)'. {\em Proceedings and transactions of the Royal Society of Canada Sectiın III (3)} {\bf 50} (1956), 17-21.
\bibitem{feld} J. FELDVOSS, 'Leibniz algebras as nonassociative algebras', Nonassociative Mathematics and its Applications, Contemporary Mathematics {\bf 721} (2017), 115-149.
\bibitem{jp} M. JIBLADZE and T. PIRASHVILI, 'Lie theory for symmetric Leibniz algebras', {\em J. Homotopy and Related Structures} {\bf 15} (2020), 167-183.
\bibitem{loday1} J.-L. LODAY, `Une version non commutative des algèbres de Lie: les algèbres de Leibniz'. {\em Enseign. Math. (2)} {\bf 39 (3–4)} (1993), 269--293.
\bibitem{loday2} J.-L. LODAY and  T. PIRASHVIL, `Universal enveloping algebras of Leibniz algebras and (co)homology', {\em Math. Annalen} {\bf 296 (1)} (1993) 139--158.
\bibitem{msy}  K.C. MISRA, E. STITZINGER and X. YU, `Subinvariance in Leibniz algebras', {\em J. Algebra} {\bf 567 (1)} (2021), 128-138.
\bibitem{stewart} I.N. STEWART, `Subideals of Lie Algebras', Ph.D. thesis, University of Warwick, 1969.
\bibitem{david 1} D. A. TOWERS, “A Frattini Theory for Algebras”. Proc.London Math.Soc 27 (1973), pp. 440– 462.
\bibitem{david 2} D. A. TOWERS, “C-ideals of Lie algebras”. Communications in Algebra 37 (2009), pp. 4366– 4373.
\bibitem{tc} D.A. TOWERS and Z. CILOGLU, `Weak $c$-ideals of a Lie algebra', {\em Turk. J. Math.}, {\bf 45} (2021), 1940-1948.
\bibitem{bnt} B.N. TURNER, `Some criteria for solvable and supersolvable Leibniz algebras, Ph.D. thesis, North Carolina State University, 2016, http://www.lib.ncsu.edu/resolver/1840.16/11085.
\bibitem{Zhu et al} L. ZHU, W. GUO and K.P. SHUM. `Weakly c-normal subgroups of finite groups and their properties'. {\em Comm. Alg.} {\bf 30} (2002), 5505-5512.
\end{thebibliography}
\end{document}